\documentclass[11pt]{article}
\usepackage{amsmath,amssymb,amsthm,amsfonts,amstext,amsbsy,amscd}
\usepackage{graphics}
\usepackage{graphicx}
\usepackage{subfigure}
\usepackage[numbers,nonamebreak]{natbib}
\usepackage[colorlinks,citecolor=blue,urlcolor=blue]{hyperref}
\usepackage{hypernat}
\usepackage{url}
\usepackage{multirow}
\usepackage{slashbox}
\usepackage{pict2e}
\usepackage{booktabs}
\usepackage{epsfig}
\makeatletter

\newskip\@bigflushglue \@bigflushglue = -100pt plus 1fil

\def\bigcentering{\let\\\@centercr\rightskip\@bigflushglue%
\leftskip\@bigflushglue
\parindent\z@\parfillskip\z@skip}

\makeatother
\numberwithin{equation}{section}
\newtheorem{satz}{Satz}[section]

\newtheorem{prop}[satz]{Proposition}
\newtheorem{theorem}[satz]{Theorem}
\newtheorem{proposition}[satz]{Proposition}

\newtheorem{lemma}[satz]{Lemma}
\newtheorem{assumption}[satz]{Assumption}

\DeclareMathOperator{\E}{{\mathbb E}}

\DeclareMathOperator{\PP}{{\mathbb P}}

\renewcommand{\phi}{\varphi}

\renewcommand{\le}{\leqslant}

\newcommand{\footnoteremember}[2]{
  \footnote{#2}
  \newcounter{#1}
  \setcounter{#1}{\value{footnote}}
}

\let\epsilon=\varepsilon
\let\ep=\epsilon
\let\phi=\varphi
   
\let\de=\delta

\let\tilde=\widetilde

\newcommand{\Sph}{\mathbb{S}^2}
\newcommand{\SO}{\mathcal{SO}_3}
\newcommand{\pje}{\boldsymbol{\psi}_{j,\eta}}
\newcommand{\pha}{\boldsymbol{\psi}_{h,\alpha}}
\newcommand{\bje}{\boldsymbol{\beta}_{j,\eta}}

\newcommand{\hbje}{\widehat{\bs{\beta}}_{j,\eta}}
\newcommand{\ezj}{\eta\in\mathcal{Z}_j}

\newcommand{\bs}{\boldsymbol}
\newcommand{\ind}[1]{\bs{1}^{}_{ \left\{#1\right\}}}
\newcommand{\Sum}{\displaystyle \sum}
\newcommand{\Sup}{\displaystyle \sup}
\newcommand{\Sj}{\Sum_{\eta\in\mathcal{Z}_j}}

\begin{document}

\abovedisplayskip=5pt
\belowdisplayskip=5pt
\title{
Application of second generation wavelets to blind spherical deconvolution}
\author{T. Vareschi \footnoteremember{adr}{Universit\'e Denis Diderot Paris 7 and CNRS-UMR 7099, 175 rue du Chevaleret 75013 Paris, France. E-mail: thomas.vareschi@univ-paris-diderot.fr}}

\date{}
\maketitle

\begin{abstract}

We adress the problem of spherical deconvolution in a non parametric statistical framework, where both the signal and the operator kernel are subject to error measurements. After a preliminary treatment of the kernel, we apply a thresholding procedure to the signal in a second generation wavelet basis. Under standard assumptions on the kernel, we study the theoritical performance of the resulting algorithm in terms of $L^p$ losses ($p\geq 1$) on Besov spaces on the sphere. We hereby extend the application of second generation spherical wavelets to the blind deconvolution framework \cite{KPP}. The procedure is furthermore adaptive with regard both to the target function sparsity and smoothness, and the kernel blurring effect. We end with the study of a concrete example, putting into evidence the improvement of our procedure on the recent blockwise-SVD algorithm \cite{DHPV}.
\end{abstract}


\section{Introduction}
\subsection{Statistical framework}
\label{Statistical framework}
Consider the following problem : we aim at recovering a signal $\boldsymbol{f}\in L^2(\Sph)$. $\boldsymbol{f}$ is not observed directly, but through the action of a blurring process modeled by a linear operator $\bs{K}$. To this end, we consider the classic white noise model, where the available information is the noisy version 
\begin{align}
\label{Signal observation}
\bs{g}_\ep=\bs{K}\bs{f}+\ep\bs{\dot{W}}
\end{align}
of $\bs{f}$, where $\bs{\dot{W}}$ is a white noise on $\Sph$ and $\bs{K}:L^2(\Sph)\to L^2(\Sph)$ is a measurable operator. We further restrict the shape of $\bs{K}$ by assuming that $\bs{K}$ is a convolution operator on $L^2(\Sph)$, a classic framework (\cite{Healy1998}, \cite{KK} and \cite{KPP}) enjoying convenient mathematical properties (see Part \ref{Harmonic analysis}). This model is equivalently formulated in a density estimation framework, in which one aims at recovering the density $\bs{f}$ of a random variable $X$ on $\Sph$ from a $n$-sample $(\bs{\theta}_1X_1,...,\bs{\theta}_n X_n)$ of $Z=\theta X$ (with the analogy $\ep\sim n^{-1/2}$), where $\bs{\theta}$ is a random element in the group of $\SO$ with density $\bs{h}_\theta$, and $Z$ has a density $\bs{f}_Z\in L^2(\Sph)$. In practice, the blurring operator $\bs{K}$ is seldom directly observable and is itself subject to  measurement errors. This covers the cases where either $\bs{K}$ is unknown but approximated via preliminary inference, or $\bs{K}$ is known but always observed with noise for experimental reasons. The result is a noisy version $\bs{K}_\de$, satisfying 
\begin{align}
\label{Operator observation}
\bs{K}_\de=\bs{K}+\de \bs{\dot{B}}
\end{align}
where $\bs{\dot{B}}$ is a gaussian white noise on $L^2(\SO)$, independent from $\bs{\dot{W}}$.
\\The relevance of this generic setting was adequately discussed in \citet{EK} and \citet{HR}, and covers numerous fields of applications. Let us mention, for example, image processing, a field which covers astronomy as well as electronic microscopy where an image, assimilated to a function $\bs{f}\in L^2([0,1]^2)$ is observed through its convolution with the Point Spread Function of the measuring device, which hence requires to be estimated in first instance (see \cite{microscopy},\cite{BID}).
\\ For $u,u',v,v',w,w'\in L^2(\Sph)$, observable quantities obtained from \ref{Signal observation} and \ref{Operator observation} hence take the form $\langle Kf,u\rangle+\ep \alpha(u)$ (signal) and $\langle \bs{K}u,v \rangle +\de \beta(v,w)$ (operator) where $\alpha(u)\sim\mathcal{N}(0,\|u\|_2)$, $\beta(v,w)\sim\mathcal{N}(0,\|v\|_2\|w\|_2)$ and $\E[\alpha(u) \alpha(u')]=\langle u,u' \rangle_{L^2(\Sph)}$, $\E[\beta(v,w) \beta(v',w')]=\langle v,v' \rangle_{L^2(\Sph)} \langle w,w' \rangle_{L^2(\Sph)}$. 
 \\As we stated, we deal with a convolution on the $2$-dimensional sphere. Namely, if Z admits a density $\bs{h}$ on $\SO$ with respect to the Haar measure, then $\bs{Kf}$ has the following expression
\begin{equation} \bs{K}\bs{f}(\omega)=\displaystyle{\int_{\SO}} \bs{f}(g^{-1}\omega)\bs{h}(g)dg
\end{equation}
where $dg$ is the Haar measure on $\SO$. That is, $\bs{f}$ is averaged on a neighbourhood of $\omega$ with weight $\bs{h}(g)$ for each rotation $g^{-1}$ applied to $\omega$. This problem, together with the introduction of needlets, is for example well illustrated by the study of ultra high energy cosmic rays (UHECR).
\\An UHECR is a radiation hitting the earth with very high energy, and whose physical origin is still unknown. Yet the understanding the mechanisms at work in this phenomenon is fundamental. Current hypothesis involve pulsars, hypernovaes or black holes. Robust statistic tools are heavily required, in order to properly estimate the density shape of the radiation, which is highly related to the physical processes at stake in its formation. One could ask, for example, whether the density is uniformly distributed among the sphere, indicating a cosmological cause, or if it is the superposition of localized spikes. In the latter case, it is crucial to determine precisely the positions of this spikes. In practice however, observations $(X_1,...,X_n)$ of such radiations are often subject to various physical perturbations, translated through the impulse response of the measuring device. We modelize these by a random rotation $\bs{\theta}$, which is to say we actually observe $(\bs{\theta}_1X_1,...\bs{\theta}_nX_n)$ realisations of the random variable $Z=\theta X$. The difficulty of the problem is characterized by the spreading of $\bs{h}_\theta$ around the identity : the less localized it is, the more difficult the estimation of $\bs{f}$ should be. Moreover, the law of $\bs{\theta}$ is not known in general, even if some assumptions can restrict its shape. In this case, preliminary inference is necessary, and leads to an estimator $\bs{K}_\de$ of $\bs{K}$.

\subsection*{Case of a known operator}

We shall concentrate here on the case where $\de=0$, and describe the path which finally led to the introduction and use of needlets in this setting. Spherical harmonics constitute the most natural set of functions to expand a target function $\bs{f}\in L^2(\Sph)$, and present a structure highly compatible with deconvolution problems. It prompted \citet{Healy1998} to solve the deconvolution problem  with their use, hereby reaching optimal $L^2$ rates of convergence (\citet{KK}). Unfortunately their performances can prove quite poor in general cases, since they lack localization in the spatial domain (see \cite{CFG}). More recently, spherical wavelets were introduced (\citet{SS} and \citet{NW1996}) and have found various applications for a direct estimation of $\bs{f}$, including  geophysics or atmospherics sciences (see for example \citet{FGS} or \citet{FMM}). However, these wavelets, which rely on a spatial construction, have an infinite support in the frequency domain, and hence are not suited for the case of spherical deconvolution, unlike spherical harmonics. The solution to this problem was brought by \citet{NW1996}, who introduced a new set of functions, called needets, which preserve the frequency localization of spherical harmonics as well as the compatibility with inverse problems, all the while remedying their lack of spatial localization. Since then, needlets became widely used in astrophysics (\citet{MPB} or \citet{CFG}) or brain shape modeling (\citet{TCGC}). In particular, \citet{KPP} reached near-minimax rates of convergence for $L^p$ losses ($1\leq p \leq \infty$) in the present spherical deconvolution setting.

\subsection*{Resolution when $\bs{K}$ is unknown : Galerkin projection}

In the case of unknown operator $\bs{K}$, the main methods involve SVD, WVD and Galerkin schemes (see \cite{CH},\cite{CR},\cite{HR} for example). We now give an overview of the so called Galerkin method and present its application to blind-deconvolution. It is based upon on a discretization of \ref{Signal observation} and \ref{Operator observation} through the choice of appropriate test functions. Suppose we want to recover $f$ from the observation $g=Kf$. Let $X,Y\subset L^2(\Sph)$ be two finite dimensional subset which admit the respective orthogonal bases $\phi=(\phi_k)_{k=1,...,n}$ and $\Phi=(\Phi_k)_{k\in 1,...,n}$. The Galerkin approximation $f_n$ of $f$ is the solution of the equation 
\begin{eqnarray} \notag \langle K f_n,v\rangle &=\langle g,v \rangle \; \forall v\in Y 
\\ \Leftrightarrow  \Sum_{k\leq n} \langle K\phi_k,\Phi_{k'}\rangle \langle f_n,\phi_k \rangle& =\langle g,\Phi_{k'} \rangle \; \forall k'\leq n
\label{Galerkin equation} \end{eqnarray}
$f_n$ is easily computable, as the equivalent solution of the finite dimensional linear system $g_n=K_n f_n$ where $g_n$ is the vector whose components are $(\langle g,\phi_k \rangle)_{k\leq n}$ and $K_n$ the matrix with entries $(\langle K \phi_k,\phi_k' \rangle)_{k,k'\leq n}$. Hence, this method relies on the discretization of the operator $K$, together with the discretization of the function $f$.
 \\Galerkin projection were successfully applied to blind deconvolution to reach optimal rates of convergence on generic Hilbert spaces (\cite{EK}) or on Besov spaces through wavelet-thresholding technics (\cite{HR}, \cite{CR}).
\\Its remains to handle two practical problems : the algorithm must include and articulate two essential steps, namely the inversion of $\bs{K}$ and the regularization of the datas through a projection/thresholding scheme . Note that both the signal and the operator $\bs{K}$ can be subject to regularization (see \cite{HR},\cite{DHPV}).
\\The second practical problem remains in chosing the right functions $\phi_k$ and $\Phi_k$. This choice should ideally answer the dilemma to find a set which is both compatible with the representation of $\bs{f}$ (through the belonging to a certain set of functions) and with the structure of $\bs{K}$ (see \cite{HR}, \cite{DHPV}). Spherical harmonics respond optimaly to the problem in the case of spherical deconvolution on Sobolev spaces for a $L^2$ error, since they realize a blockwise-SVD decomposition of $\bs{K}$, as shown in \ref{Blockwise property}. More importantly here, when $\de$ is non negative, they allow a fine treatment of $\bs{K}_\de$ thanks to the sparse structure of the original operator $\bs{K}$, which allowed \citet{DHPV} to reach optimal $L^2$-rates of convergence for a natural class of operators and functions. Thus, we should always seek to preserve this property of sparsity whenever possible.

\subsection{Harmonic analysis on \texorpdfstring{$\SO$}{SO} and \texorpdfstring{$\Sph$}{Sph}}
\label{Harmonic analysis}
The next part provides preliminary tools in order to apply a blockwise scheme to the case of spherical deconvolution. It is a quick overview of harmonic analysis on the spaces $\Sph$ and $\SO$ which is mostly inspired by \citet{Healy1998}.
Let us define the Euler matrices $$u(\phi)=\begin{pmatrix} \cos \phi & -\sin \phi & 0 \\ \sin \phi & \cos \phi &0 \\ 0 &0 & 1 \end{pmatrix} , \; a(\theta)=\begin{pmatrix} \cos \theta & 0 & \sin \theta \\  0 & 1 &0 \\ - \sin \theta  &0 & \cos \theta \end{pmatrix}$$
where $\phi \in [0,2\pi),\,\theta\in[0,\pi)$.
\\Every rotation $g$ in $\SO$ is the product of 3 elementary rotations :
\begin{equation}\ep=u(\phi) a (\theta) u(\psi)\end{equation}
where  $\phi,\psi \in [0,2\pi),\,\theta\in[0,\pi)$ are the Euler angles of $g$.
Let $l \in \mathbb{N}$ and $-l\leq m,n \leq l$. We also define the rotational harmonics
\begin{equation} R^l_{mn}(\phi,\theta,\psi)=e^{-i(m\phi+n\psi)}P^l_{mn}(\cos(\theta)) \end{equation}
where $P^l_{mn}$ are the second type Legendre functions.
\\The functions $R^l_{mn}$, $l\in\mathbb{N}$, $|m|,|n|\le l$ are the eigenfunctions of the Laplace-Beltrami operator on $\SO$, associated with the eigenvalues $2l+1$. Therefore, the system $(\sqrt{2l+1}R_{mn}^l)^l_{mn}$ forms a complete orthonormal basis of $L^2(\SO)$.
Let $h\in L^2(\SO)$. For all $l \geq 0$, the projection of $h$ on the space of rotational harmonics with degree $l$ is
$$\Sum_{m,n=-l}^l\hat{h}^l_{mn} R_{mn}^l$$
where $\hat{h}^l_{mn}$ is the $(l,m,n)$ Fourier coefficient of $h$, defined by
\begin{equation}
\hat{h}^l_{mn}=\int_{\SO} h(g) \overline{R^l_{mn}(g)} d\mu g
\end{equation}
An analogous study is available on $\Sph$. Any point $\omega \in \Sph$ is determined by its spherical coordinates $(\theta,\phi)$:
\begin{equation}
\omega=(\sin(\theta)\cos(\phi),\sin(\theta)\sin(\phi),\cos(\theta))
\end{equation}
 where $\theta \in [0,\pi)$ and $\phi \in [0,2\pi)$.
Let  $l$ a positive integer, $m,n$ two integers ranking from $-l$ to $l$. Define the following functions, known as the spherical harmonics, on $\Sph$ :
\begin{equation}
Y_m^l=(-1)^m \sqrt{\frac{2l+1}{4\pi}\frac{(l-m)!}{(l+m)!}}P^l_m(\cos(\theta))
\end{equation}
where $P^l_m$ are the Legendre functions. The set $(Y_m^l)$ constitutes an orthonormal basis of $L^2(\Sph)$. Note $\mathbb{H}_l$ the space of spherical harmonics of degree $l$ and $\bs{P}_l$ the orthogonal projector onto  $\mathbb{H}_l$. For every function $f\in L^2(\Sph)$, 
$$\bs{P}_l f=\Sum_{m=-l}^l \hat{f}_m^l Y_m^l$$ 
where $\hat{f}_m^l$ is the $(l,m)$ Fourier coefficient of $f$, defined by 
$$\hat{f}_m^l=\displaystyle{\int_{\Sph}} f(\omega) \overline{Y_m^l(\omega)} d\omega$$
The term of Blockwise-SVD finds its roots in the following proposition, which expresses the link between Fourier coefficients of $h\ast f$ and those of $h$ and $f$. A proof is present in \cite{Healy1998}.
\begin{proposition}[Blockwise property]
\label{Blockwise property}
Let $h\in L^2(\SO) $ and $f\in L^2(\Sph) $ The Fourier coefficients of $h\ast f$ are
\begin{align*}
(\widehat{h\ast f})_m^l= \Sum_{n=-l}^l \hat{h}_{mn}^l \hat{f}_n^l=\Sum_{n=-l}^l \langle h\ast Y_n^l,Y_m^l \rangle \hat{f}_n^l
\end{align*}
\end{proposition}
Hence, if K is a convolution operator over $L^2(\Sph)$ and $f\in L^2(\Sph)$, and if we note, $K^l$ the matrix $\big(\langle K Y_n^l,Y_m^l\rangle\big)_{|m|,|n|\leq l}\in M_{2l+1}(\mathbb{C})$ and $f^l$ the vector $(\langle f,Y_m^l\rangle)_{| m|\leq l} $, Proposition \ref{Blockwise property} translates into $\big(Kf\big)^l= K^l f^l$. Hence, turning back to the Galerkin projection of $\bs{K}$, take $\phi=\Phi=(Y_m^l)_{m,l}$, $| m|\leq l$. Proposition \ref{Blockwise property} actually implies that the Galerkin matrix $\big(\langle\bs{K} Y_{m_1}^{l_1},Y_{m_2}^{l_2}\rangle\big)_{l_i\geq 0,| m_i|\leq l_i,\,i=1,2}$ is sparse, with blocks $\bs{K}^l$ on its diagonal. This justifies the denomination of blockwise-SVD decomposition. In the sequel, if $f\in L^2(\Sph)$ and $K:L^2(\Sph)\to L^2(\Sph)$ is a convolution operator, we will refer indifferently to $\bs{P}_lf$ or $f^l$, and to $\bs{P}_l K \bs{P}_l$ or $K^l$. Besides, due to Parseval's formula, we also have $\|\bs{P}_lf\|_{L^2(\Sph)}=\|f^l\|_{\ell^2(\mathbb{C}^{2l+1})}$ and $\|\bs{P}_l K \bs{P}_l\|_{L^2(\Sph)\to L^2(\Sph)}=\|K^l\|_{\ell^2(\mathbb{C}^{2l+1}) \to \ell^2(\mathbb{C}^{2l+1})}$.
Turning back to the original problem and reminding Proposition \ref{Blockwise property}, we can reformulate the equivalent problem, obtained by projecting \ref{Signal observation} and \ref{Operator observation} on every space $\mathbb{H}_l$:
\begin{align}
\forall l \geq 0,\;\bs{g}_\ep^l=&\bs{K}^l \bs{f}^l + \ep \dot{\bs{W}}^l
\\ \forall l \geq 0,\;\bs{K}_\de^l=&\bs{K}^l+\de \dot{\bs{B}}^l
\end{align}
 where $\dot{\bs{W}}^l$ is a centered gaussian vector with covariance $\bs{I}_{2l+1}$, and $\dot{\bs{B}}^l$ is a $(2l+1)\times(2l+1) $ matrix whose entries are iid variables with common law $\mathcal{N}(0,1)$.\\
As we said, spherical harmonics however show great inconvenients when used in the estimation of a generic function $\bs{f}\in L^2(\Sph)$. We turn to the presentation of far more accute functions to this end.

\section{Needlets}\label{Needlets}
\subsection{Construction of Needlets}
Needlets were introduced in \citet{NPW}, and used in the framework of density estimation on the sphere by \citet{KP}, \citet{BKMP} and \citet{KPP}. As their construction relies on a rearrangement of spherical harmonics, they inherit the very useful stability properties of the latter in inverse problems (as expressed in \ref{Blockwise property}). In addition, whereas spherical harmonics' supports spread all over the sphere, needlets are almost exponentially localized around their respective centers, thus allowing a fine multi-resolution analysis and a description of very general regularity spaces on $\Sph$.

\subsection*{Needlet framework}

As we have seen, the following decomposition holds : $L^2(\Sph)=\displaystyle{\bigoplus_{l=0}^{\infty}}\mathbb{H}_l$.
The orthogonal projector $\bs{P}_l$ on $\mathbb{H}_l$ can be written 
$$\bs{P}_l(f)=\int L_l( \langle x,y \rangle )f(y)dy=\int\Sum_{m=-l}^l Y_m^l(x) \overline{Y_m^l(y)}f(y)dy$$
where $L_l$ is the Legendre polynomial of degree $l$, and $\langle.\rangle$ stands for the usual scalar product on $\mathbb{R}^2$. Finally, the fact that $\bs{P}_l$ is a projector implies the identity
\begin{equation}
\label{Reproduction property}
\displaystyle{\int_{\Sph}}L_l(\langle x, y \rangle) L_k(\langle y,z \rangle) dz = \delta_{l,k} L_l (\langle x,z\rangle)
\end{equation}

\subsection*{Littlewood-Paley decomposition}

Let $a$ be a $\mathcal{C}^\infty(\mathbb{R})$ symetric function, compactly supported in $[-1,1]$, decreasing on $\mathbb{R}^+$, such that for all $x\in\mathbb{R}$, $0\leq a(x) \leq 1$ and for all $|x|\leq 1/2$, $|a(x)|=1$. 
Define , for all $x\in\mathbb{R}$, $$b^2(x) = a(\frac{x}{2}) -a(x) $$
$b^2$ is a positive function, supported in $[-2;-1/2]\bigcup[1/2;2]$, satisfying by construction
\begin{equation} \forall |x|\geq 1,\,\Sum_{j\geq0} b^2(\frac{x}{2^j})=1 
\end{equation}
Define the kernels  
\begin{align}
\Lambda_j(x,y)=\Sum_{l\geq0} b^2(\frac{l}{2^j}) L_l(\langle x,y \rangle)\text{, and }M_j(x,y)=\Sum_{l\geq0} b(\frac{l}{2^j}) L_l(\langle x,y \rangle)
\label{Def kernels}
\end{align} and the associated operators
\begin{align}
\label{Aj Bj}
B_j f=\displaystyle{\int_{\Sph}} \Lambda_j(x,y)f(y)dy
\text{ and }A_J f=\Sum_{j=-1}^J B_j f
\end{align}
with the convention $B_{-1} f=\bs{P}_0 f$. Note that the two sums in \ref{Def kernels} are finite since $b(\frac{l}{2^j})=0$ if $l\notin L_j$, where $L_j$ is the set of integers between $2^{j-1}$ and $2^{j+1}-1$. It is straightforward to show that, for all $f\in L^2(\Sph)$,
\begin{align}
\label{L2 behaviour of needlets}
\|f\|_2^2=\Sum_{j\geq 0}\Sum_{\ezj} \langle f,\pje \rangle^2
\end{align}
One of the main results in \citet{NPW2006} is that $A_J$ also mimicks the best polynomial approximation of $f$ with respect to $\|.\|_p$ for all $p\geq 1$, as expressed in the following theorem:
\begin{theorem}
For all $p\in[1,\infty[$, if $f\in L^p(\Sph)$, then $$\|A_J f-f\|_p\to 0 \text{ as } J\to \infty$$, with uniform convergence if $f\in \mathcal{C}^{0}(\Sph)$.
\end{theorem}
\subsection*{Space discretization}
\begin{proposition}[Quadrature formula].
\label{Quadratute formula}
 Note $\mathcal{P}_l$ the set of polynoms with degree less than $l$ on $\Sph$. For each $l \geq 0$, there exists a finite set $Z_l$ of cubature points, and non negative reals $(\lambda_\eta)_{\eta\in Z_l}$ such that
$$ \forall f \in \mathcal{P}_l,\, \displaystyle{\int_{\Sph}} f(\omega)d\omega = \Sum_{\eta\in Z_l} \lambda_\eta f(\eta) $$
\end{proposition}
Since $b(\frac{l}{2^j})\neq 0$ only if $2^{j-1} \leq l <2^{j+1}$ the function $z\mapsto M_j(x,z)$ belongs to $\mathcal{P}_{2^{j+1}-1}$, and  $z\mapsto M_j(x,z)M_j(z,y)$ is an element of $\mathcal{P}_{2^{j+2}-2}$. For more convenience, we will note $Z_{2^{j+2}-2}=\mathcal{Z}_j$. Hence, $B_j$ writes
\begin{align*}
B_j(f)&=\displaystyle{\int_{\Sph}} \big(\Sum_{\ezj} \lambda_\eta M_j(x,\eta)M_j(\eta,y) dz\big) f(y)dy
\\&=\Sum_{\ezj}  \sqrt{\lambda_\eta} M_j(x,\eta) \displaystyle{\int_{\Sph}}  \sqrt{\lambda_\eta} M_j(\eta,y) f(y)dy
\end{align*}
The functions $\pje=\sqrt{\lambda_\eta} M_j(.,\eta)$, are called needlets. Furthermore, it can be prooved that the cubature points $\eta$ and weights $\lambda_\eta$ can be chosen so that the two following conditions are verified,
\begin{equation}
c^{-1}2^{2j} \leq\mathrm{card}(\mathcal{Z}_j) \leq c2^{2j} \text{ and } c^{-1}2^{-2j}\leq \lambda_\eta \leq c2^{-2j}
\end{equation}
with a constant $c>0$
\subsection{Besov spaces}
\subsubsection*{Properties of needlets}
By construction, needlets are well localized in frequence ($\mathcal{C}^\infty$, compactly supported). A crucial result proved by \citet{NPW2006} shows that they are furthermore near exponentially localized in space.  
\begin{theorem}
Let $j\geq 0,\, \ezj$. For all $M>0$, there exists $C_M>0$ such that 
\begin{equation}
\forall x \in \Sph,\, |\pje(x)|\leq \frac{C_M 2^j}{(1+2^jd(x,\eta))^M}
\end{equation}
\end{theorem}
where $d(x,y)=\arccos(\langle x,y \rangle)$ is the geodesic distance on the sphere. To illustrate this point, we represented two needlets of level $j=2,3$ on figure \ref{Needplot}. The following properties are all consequences of this localization property.
\begin{figure}
\centering
\begin{tabular}{cc}
\includegraphics[width=5cm]{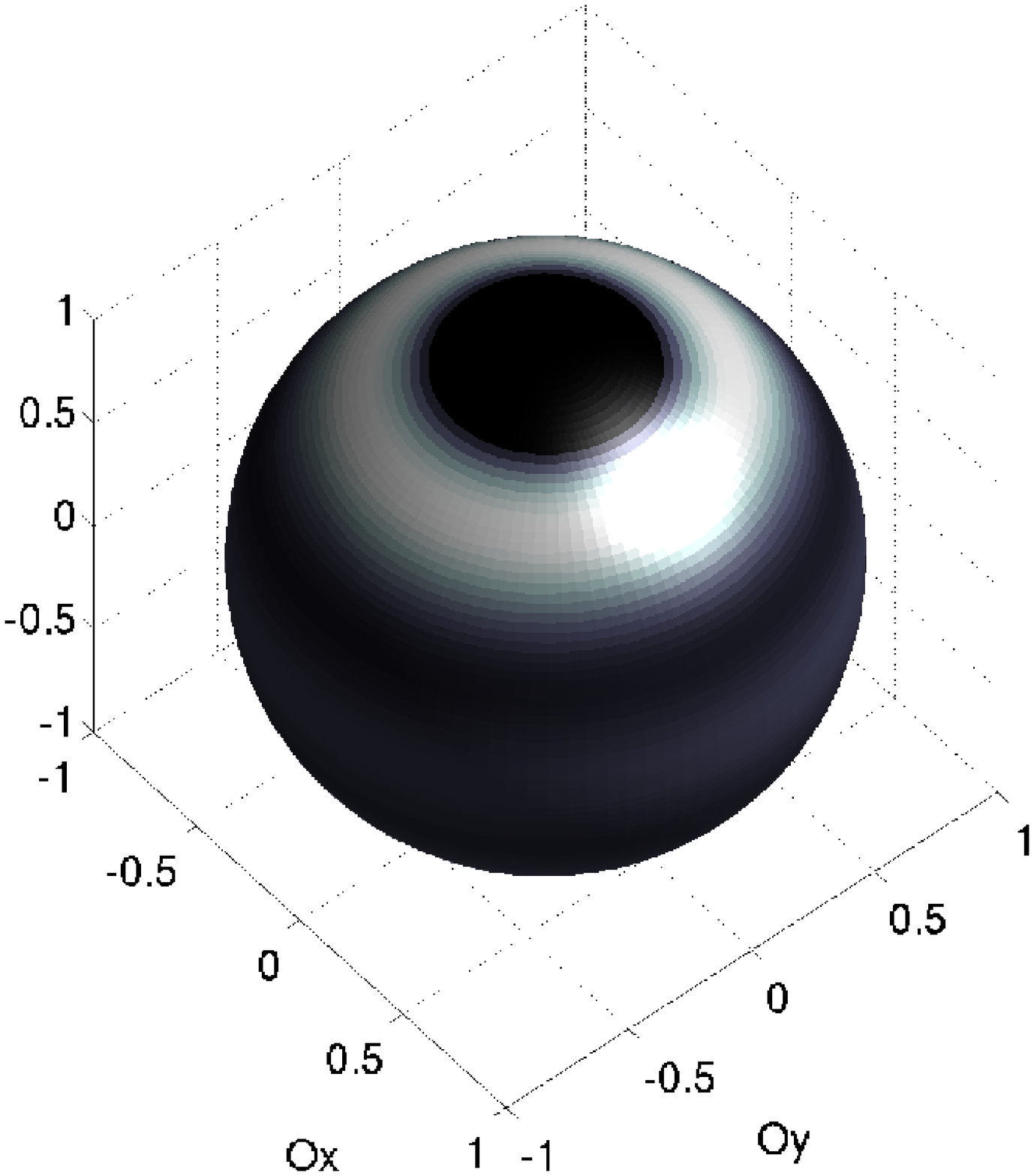}&\includegraphics[width=5cm]{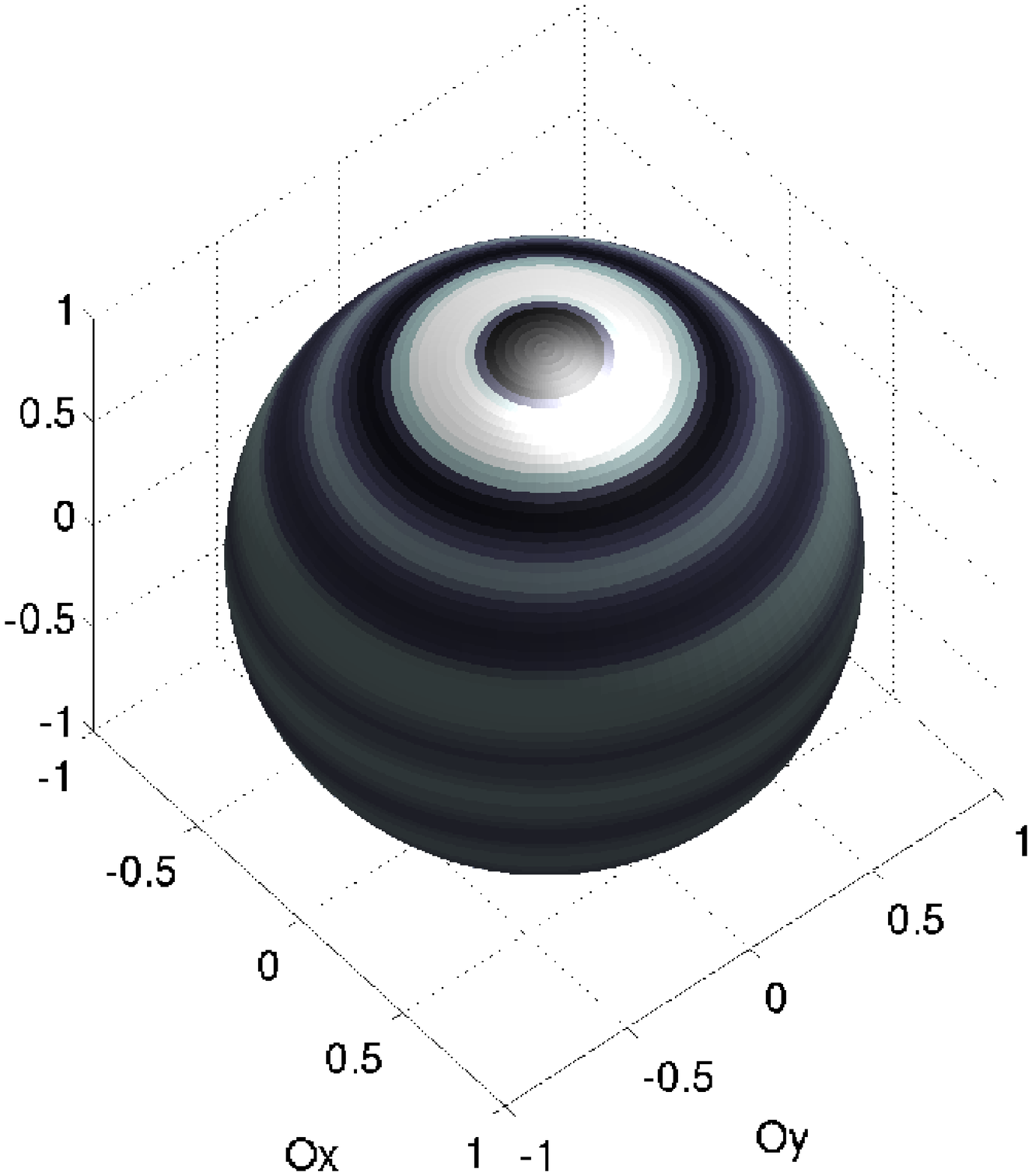}
\end{tabular}
\caption{\footnotesize A Spherical representation of two needlets (level $j=2,3$ from left to right) centered around the point (0,0,1). The darkened zones correspond to the regions where the needlet is high.}
\label{Needplot}
\end{figure}
\begin{proposition}
\label{Needlets properties}
For all $p\geq 1$ (with the convention $1/\infty=0$), there exists $c_p,C_p,D_p>0$ such that
\begin{equation}
\label{Needlets norm}
c_p 2^{j(\frac{1}{2}-\frac{1}{p})} \leq \|\pje\|_p \leq C_p 2^{j(\frac{1}{2}-\frac{1}{p})}
\end{equation}
\end{proposition}
\begin{prop}
\label{Lp behaviour of needlets}
For all  $p\in[1,+\infty]$, there exists a constant $C_p$ such that for all $f\in L^p(\Sph)$,
\begin{align}
\|B_j(\bs{f})\|_p \leq C_p \big\| \big(|\lambda_\eta| \|\pje\|_p \big)_{\ezj}\big\|_{\ell^p} &
\\ \text{Moreover, } \big\| \big(|\lambda_\eta| \|\pje\|_p \big)_{\ezj}\big\|_{\ell^p} \leq \|f\|_p&
\end{align}
\end{prop}
\subsubsection*{Construction of Besov spaces}
Besov spaces on the sphere naturally generalize the usual approximation properties of regular functions, all the while being simply characterized with the help of needlets. A complete description, and the proofs of the results claimed in this part can be found in \citet{NPW} or \citet{KP}.
Let $f:\Sph \mapsto \mathbb{R}$ be a measurable function and let $E_{k,\pi}$ ($\pi \geq 1$) be the distance of $f$ to $\mathcal{P}_k$ with respect to $\|.\|_{L^\pi}$, that is 
$$E_{k,\pi}= \displaystyle{\inf_{P\in \mathcal{P}_k}} \|f-P\|_{L^\pi} $$
\begin{theorem}
Let $0<s<\infty$, $1 \leq p < \infty$ and $0<r\leq \infty$.Let $f\in L^\pi$. The following statements are equivalent and define the Besov space $B_{\pi,r}^s$.
\begin{align}
\big( \Sum_{k\geq 0} k^{rs} E_{k,\pi}(f)^r\big)^{1/r}&<\infty 
\\ \big( \Sum_{j\geq 0} 2^{jrs} E_{2^j,\pi}(f)^r\big)^{1/r}&<\infty 
\\ \exists\,\xi_j\in \ell^q(\mathbb{N}), \; \|B_j f\|_\pi&=\xi_j  2^{-js} 
\\ \exists \, \xi_j\in \ell^q(\mathbb{N}),\; \big(\Sum_{\eta\in \mathcal{Z}_j} |\beta_{j,\eta}|^\pi  \| & \pje\|_\pi^\pi \big)^{1/\pi}=\xi_j 2^{-js}
\end{align}
\end{theorem}
$B_{\pi,q}^s$ is a Banach space, associated with the norm 
$$\|f\|_{B_{\pi,r}^s}=\|2^{j(s+2(\frac{1}{2}-\frac{1}{\pi}))} (\Sj |\bje|^\pi)^{1/\pi} \|_{\ell^r}  $$
Besov spaces satisfy the following includings, all of which derive from H\"older's inequality.
\begin{prop}{Besov embeddings.}
\label{Besov embeddings}
Let $s>0$, $1\leq p,\pi,r \leq \infty $
\begin{itemize}
\item $B_{\pi,r}^s \subset B_{p,r}^s$ if $\pi\geq p$.
\item $B_{\pi,r}^s \subset B_{p,r}^{s-2(\frac{1}{\pi}-\frac{1}{p})}$ if $\pi<p$ and $s-2(\frac{1}{\pi}-\frac{1}{p})>0 $
\item $B_{\pi,r}^s \subset \mathcal{C}^0(\Sph)$ if $s>\frac{2}{\pi}$, where $\mathcal{C}^0(\Sph)$ is the set of continuous functions on $\Sph$.
\end{itemize}
\end{prop} 
\section{Estimation procedure}
We turn to the presentation of our procedure of Blind Deconvolution using Needlets (\textbf{BND}) and derive rates of convergence for generic $L^p$ losses on Besov spaces. A natural idea would be to take needlets as test functions in equation \ref{Galerkin equation} since they represent $\bs{f}$ efficiently. Unfortunately, the ensuing Galerkin matrix $\big(\langle \bs{K} \pje,\pha \rangle\big)_{j\geq 0,\ezj,h\geq 0,\alpha\in\mathcal{Z}_h}$ has many non-zero entries, due to the fact that the frequency levels of $\pje$ and $\psi_{h,\alpha}$ overlap if $|j-h|\leq 1$. The choice of the functions $Y_m^l$ is far more indicate, moreover the ensuing matrices $\bs{K}^l$ enter naturally in the needlets decomposition of $\bs{g}_\ep$ since, with the use of Parseval's formula, we have $$\langle \bs{g}_\ep ,\pje \rangle = \Sum_{l\in L_j}\langle \bs{K}^l \bs{f}^l+\ep\dot{\bs{W}}^l,\pje^l \rangle  $$
Before entering into details in the procedure, we need to precise the blurring effect of $\bs{K}$ with the introduction of a constant $\nu$ called degree of ill-posedness (DIP) :
\begin{assumption}[Spectral behaviour of K]
\label{Spectral behaviour of K}
There exists $\nu\geq0$, $Q_1(\bs{K}),Q_2(\bs{K})\geq 0$ such that, for all $l \in \mathbb{N}^*$, 
\begin{align}
 \label{DIP K-1} Q_1 l^{\nu}\leq \|(\bs{K}^l)^{-1} \|_{op} &\leq Q_2 l^{\nu}
\end{align}
We note $\mathcal{K}_\nu(Q_1,Q_2)$ the set of operators satisfying this assumption.
\end{assumption}
 Assumption \ref{Spectral behaviour of K} actually states that even if $\bs{K}$ is $L^2$ continuous, its inverse is not bounded and hence not computable in a satisfying way, but the weaker assumption that $\bs{K}:\mathcal{W}^{-\nu/2}\to \mathcal{W}^{\nu/2}$ is continuous holds (see \cite{NP}).\\
Let us now give an intuition of the procedure. Decomposing the inner product $\langle \bs{Kf},\pje \rangle$, $j\geq 0,\, \ezj$ on every space $\mathbb{H}_l,\,l\geq 0$ via Parseval's formula, coupled to Proposition \ref{Blockwise property} entails
\begin{align*}
\langle \bs{f},\pje \rangle =\Sum_{l\in L_j} \langle (\bs{K}^l)^{-1} (\bs{Kf})^l,\pje^l\rangle
\end{align*}
Hence a first natural estimator of $\langle \bs{f},\pje \rangle$ would be
\begin{align*}
\tilde{\bje}=\Sum_{l\in L_j} \langle (\bs{K}_\de^l)^{-1} \bs{g}_{\ep}^l,\pje^l\rangle
\end{align*}
Remark that the elements $\pje^l,\, l\in L_j$ are easily computable thanks to the identity  
\begin{align*}
\langle \pje,Y_m^l\rangle= b(\frac{l}{2^j}) \overline{Y_m^l(\eta)}\text{ for all }l\in L_j,|m|\leq l
\end{align*}
However, the presence of noises on both the signal and operator requires an additional treatment. This is realized through a preliminary processing $T_{op}(\bs{K}^l)$ of $\bs{K}^l$ and a secondary treatment $T_{sig}\big(\Sum_{l\in L_j} \langle \big(T_{op}(\bs{K}^l)\big)^{-1} \bs{g}_\ep^l,\pje^l\rangle \big)$ of the resulting estimator.

\subsection{Main procedure}

Suppose that Assumption \ref{Spectral behaviour of K} holds. Define $J$, the maximal resolution level, such that 
\begin{equation}
2^J=\lambda \lfloor \big(\ep \sqrt{|\log\ep|}\big)^{-1} \wedge \big(\de\sqrt{|\log\de|} \big)^{-2} \rfloor
\label{Max level}
\end{equation} 
for a positive parameter $\lambda$.  For $j\in \mathbb{N}$, define \begin{align*}l_j=\min\{l\in L_j, \; \| T_{op}\big(\bs{K}^l\big)^{-1}\|\neq 0\}
\end{align*} (with the convention $\min \emptyset= +\infty$), and, for positive constants $\kappa$ and $\tau_{sig},\tau_{op}$,
{\small
\begin{align}
 O_{l,\de}&=\kappa \sqrt{2l+1} \de \sqrt{|\log\de|} \label{Operator thresholding}
\\S_j(\de,\ep)&=\begin{cases} \|T_{op}(\bs{K}^{l_j})^{-1}\|_{op}\Big(\tau_{sig}\ep \sqrt{|\log \ep|} \vee \tau_{op} 2^{-j/2}\de\sqrt{|\log \de|} \Big) &\text{ if } l_j<\infty
\\+\infty &\text{ if } l_j=+\infty
\end{cases} \label{Signal thresholding}
\end{align}
}
Define also the ensuing regularizing procedures $T_{sig}$ and $T_{op}$, inspired from \cite{DHPV} and \cite{KPP}, defined by
\begin{align*}
\forall g\in L^2(\Sph),\;T_{sig}(g)&=\Sum_{j=0}^J \Sum_{\eta\in \mathcal{Z}_j} \langle g,\pje \rangle  \ind{|\langle g,\pje \rangle|>S_j(\de,\ep)} \pje
\\ \forall K\in L^2(\SO),\; T_{op}(K)&=\Sum_{l=0}^{2^{J+1}} \bs{K}^l \ind{\|(\bs{K}^l)^{-1}\|\leq O_{l,\de}^{-1}}
\end{align*}
The estimator $\tilde{f}$ of $f$ is defined by 
\begin{align*}
\tilde{\bs{f}}&= T_{sig}\Big( \big(T_{op}(\bs{K}_\de)\big)^{-1}\bs{g}_\ep\Big)
\\&=\Sum_{j\leq J} \Sum_{\ezj} \hbje \ind{|\hbje|>S_j(\de,\ep)} \pje 
\end{align*}
where we noted $\hbje\overset{\Delta}{=}\Sum_{l=2^{j-1}}^{2^{j+1}}\langle (\bs{K}_\de^l)^{-1} \ind{\|(\bs{K}_\de^l)^{-1}\|\leq O_{l,\de}^{-1}}\bs{g}_\ep^l ,\pje^l \rangle $.
\begin{theorem}
Let $\pi\geq 1$, $s>\frac{2}{\pi}$, $r\geq 1$ and $M>0$. Let $\nu\geq 0$, let $Q_1\geq Q_2>0$.
 Then for sufficiently large $\kappa$ and $\tau$, for all $p\in[1,+\infty[$,
 {\large
\begin{align}
\Sup_{f\in B_{\pi,r}^s(M),\bs{K}\in \mathcal{K}_\nu(Q_1,Q_2) }\E\|\tilde{f}-f \|_p^p \notag\lesssim  &(|\log\ep|)^{p-1} (\ep \sqrt{|\log \ep|})^{p\mu(2) } 
\\ &\vee(|\log\de|)^{p-1} (\de \sqrt{|\log \de|})^{p\mu(1)}
\end{align}
}
where $\lesssim$ means inequality up to a multiplicative constant depending only on $p,s,\pi,r,M,\nu,Q_1,Q_2,\lambda,\kappa,\tau_{sig}$ and $\tau_{op}$, and where the exponents $\mu(d)$ are defined for $d\in\mathbb{N}$ by
\begin{align*}
\mu(d)= &\begin{cases}
\frac{s}{s+\nu+\frac{d}{2}}&\text{ if } s>(\nu+\frac{d}{2})(\frac{p}{\pi}-1) \\&\text{ or }s=(\nu+\frac{d}{2})(\frac{p}{\pi}-1) \text{ and }r\leq \pi
\\
\\\frac{s-2/\pi+2/p}{s-2/\pi+\nu+\frac{d}{2}}&\text{ if } \frac{2}{\pi}<s<(\nu+\frac{d}{2})(\frac{p}{\pi}-1)
\end{cases}
\end{align*}

\label{Convergenge p fini}
\end{theorem}
\begin{theorem}
Under the same hypothesis as in Theorem \ref{Convergenge p fini},
{\large
\begin{align}
 \Sup_{f\in B_{\pi,r}^s(M),\bs{K}\in \mathcal{K}_\nu(Q_1,Q_2) }\E\|\tilde{f}-f \|_\infty \notag
\lesssim  &\sqrt{|\log\ep|} (\ep \sqrt{|\log \ep|})^{\mu'(2)} \\&\vee\sqrt{|\log\de|} (\de \sqrt{|\log \de|})^{\mu'(1)}
\end{align}
}
where the exponents $\mu'(d)$ are defined  by
$$\mu'(d)= 
\frac{s-2/\pi}{s-2/\pi+\nu+\frac{d}{2}}$$
\label{Convergenge p infini}
\end{theorem}
An explanation of the shape of the thresholding procedure is necessary here. The term $\|T_{op}(\bs{K}^{l_j})^{-1}\|_{op}$ is meant to replace the classical term $2^{j\nu}$ (see \cite{KPP}). Indeed, Lemmas \ref{Operator concentration} and \ref{Neumann} show that with high probability, this term behaves as $2^{j\nu}$. The procedure \textbf{BND} is hence adaptive for a wide range of $L^p$ losses and over a wide range of function and operator spaces, with respect to $s,\pi,r,Q_1,Q_2$, and $\nu$.
\\What can we say about the case where $\nu$ is already known or infered? First, in that case, we can directly replace $\|T_{op}(\bs{K}^{l_j})^{-1}\|_{op}$ by $2^{j\nu}$ in the threshold level \ref{Signal thresholding}. Secondly, the lower bound in Assumption \ref{Spectral behaviour of K} becomes unnecessary (i.e. we can set $Q_2=0$) and the class of operators for which the rates of Theorems \ref{Convergenge p fini} and \ref{Convergenge p infini} are available hence becomes wider. Finally, we can use a sharper maximal level $$
2^J=\lambda \lfloor \big(\ep \sqrt{|\log\ep|}\big)^{\frac{-1}{\nu+1}} \wedge \big(\de\sqrt{|\log\de|} \big)^{\frac{-1}{\nu+1/2}} \rfloor$$ which will lead to the same rates of convergence, while avoiding unnecessary calculations. This is a non negligible gain, since  needlets are costly with regard to computation time.
\\Although we chose to work in a white noise model for the convenience of calculations, the algorithm and ensuing results should be easily transcriptible to the density estimation framework, in which one observes direct realizations $(\theta_1 X_1,...,\theta_n X_n)$ of $ \theta X$ and a noisy version $\bs{K}_\de$ of $\bs{K}$.
\\More generally, the presence of a blockwise-SVD decomposition combined with properties of the ensuing needlets frame similar to Part \ref{Needlets} ensure the applicabilty of the scheme with adapted convergence rates. This includes in particular the corresponding one dimensionnal problem (equivalent to deconvolution in a periodic setting), where the rates improve on those of \citet{CR} and \citet{HR}. Another practically relevant example concerns the operators defined on $\mathbb{S}^d$, $d\geq 1$ via  $$\bs{K}f(\xi)=\int_{\mathbb{S}^d} \phi(\langle \xi,\omega\rangle )f(\omega)d\omega$$ and $\phi$ is a bounded integrable function on $[-1,1]$. In this case, as shown by the Funk-Hecke theorem (see \cite{Groemer}), spherical harmonics realize a SVD of $\bs{K}$. On the other hand, the construction of needlets generalizes naturally to $\mathbb{S}^d$ (\cite{NPW2006}), and the rates derived hence change to $$(|\log\ep|)^{p-1} (\ep \sqrt{|\log \ep|})^{\mu(d)} \vee(|\log\de|)^{p-1} (\de \sqrt{|\log \de|})^{\mu(0)} $$
This sheds a light on the role of the dimensional factors obtained in the rates, the term $\mu(d)$ accounting for the dimension of the underlying space while the term $\mu(0)$ concerns the efficiency of the set of functions chosen for the Galerkin projection, via the size of the blocks obtained.
\\The speed of convergence gives an explicit interplay between $\de$ and $\ep$, including the possible case where $\de\gg\ep$. If $\de=0$, the rates coincide with the results of \citet{KPP} (actually, the algorithms themselves are nearly identical), which are  optimal in the minimax sense (up to a log factor, see \citet{Willer} for a sketch of proof). The two regions $s\geq (\nu+1)(\frac{p}{\pi}-1)$ and $s<(\nu+1)(\frac{p}{\pi}-1)$ are classic in non parametric estimation and respectively refered to as the regular case and the sparse case.
\\The optimality (in a minimax sense) of the procedure is beyond the scope of this paper. We don't know if the $\de$-rate is minimax in general, though it is trivially the case if the $\ep$-term dominates the $\de$-term in theorems \ref{Convergenge p fini} and \ref{Convergenge p infini}. Let us point out that \citet{DHPV} attained a faster (and optimal) rate in the particular case where $p=\pi=r=2$. However, the corresponding framework was more restrictive since it (crucially) requires the set of inequalities  
\begin{align}
\|\bs{K}^l\|\lesssim l^{-\nu} \text{ and } \|\big(\bs{K}^l\big)^{-1}\|\lesssim l^{\nu}
\label{BBD assumption}
\end{align}
, which unilaterally entail Assumption \ref{Spectral behaviour of K}. Secondly, the procedure developed therein  relies strongly on the conveniency of spherical harmonics to represent both the operator and the signal sparsely, and isn't directly transcriptible in the present setting without additional restrictions on the behaviour of $\bs{K}$ (a direct transcription of the algorithm actually shows very poor practical performances).

\subsection{Practical study}

We present the practical numerical performances of \textbf{BND} and compare it to the Blind Blockwise Deconvolution algorithm (\textbf{BBD}) of \citet{DHPV}. The sets of cubature points in the simulations that follow have been taken from the web site of R. Womersley 
\url{ http://web.maths.unsw.edu.au/~rsw}.
We proceed with the following choices of parameters :
\\\textbf{Data}: the target density $\bs{f}$ is given by \begin{align*}
\bs{f}(\omega)=\exp(-2*\|\omega-\omega_1\|_{\ell^1(\mathbb{R}^3)})/c
\end{align*}
with $\omega_1=(0,1,0)$ and $c=0.6729$. Concerning the operator $\bs{K}$, we choose it among the class of Rosenthal laws on $\SO$. These distributions find their origins in random walks on groups (see \cite{Rosenthal}). $\bs{K}$ is said to follow a Rosenthal distribution of parameters $\alpha\in]0;\pi]$ and $\nu>0$ on $\SO$ if, for $l\geq0$, $|m|\leq l$, we have
\begin{align*}
\bs{K}_{m,n}^l=\Big( \frac{\sin((l+1/2)\alpha)}{(2l+1)\sin(\alpha/2)}\Big)^\nu\ind{m=n}
\end{align*}
A Rosenthal law hence provides a concrete example of operator with DIP $\nu\geq0$. We will take $\alpha=\pi$ and $\nu=1$.
\\\textbf{Tuning parameters}: we set $\lambda=1$ in \ref{Max level}. The concrete choice of adequate thresholding constants $\kappa$ and $\tau$ is a complex issue. Our practical choices will be based on the following remark, inspired from \citet{Donoho1994}: in the case of direct estimation on real line, the universal threshold which is both efficient and simple to implement, takes the form $2\sqrt{|\log \ep|}$. A consistent interpretation is to consider that this threshold should kill any pure noise signal. We will adapt this reasoning to the case of interest.
\\\underline{Choice of $\kappa$} : we use as a benchmark the case where $\bs{K}^l$ is the null matrix of $M_{2l+1}(\mathbb{R})$ for $l\geq 1$ (this corresponds to the case where the law of $\theta$ is uniform over $\SO$). Given $\de$ large enough, the smallest value $\kappa_\de$ such that , in the Fourier basis, the number of remaining levels $l\leq 10$ is zero, is hence retained. The results are reported in table \ref{Operator benchmark} and give $\kappa=0.8$.
\\\underline{Choice of $\tau_{sig}$ and $\tau_{op}$}: It is clear that the role of  $\tau_{sig}$ and $\tau_{op}$ is to control the influence of the signal (resp.  the operator) error. To chose $\tau_{sig}$ (resp. $\tau_{op}$), we therefore chose $\ep_{sig}>\de_{sig}>0$ (resp. $\de_{op}>\ep_{op}>0$) large enough. Following \citet{KPP}, we use the uniform density $\bs{u}$ on $\Sph$ as a benchmark. We have $\langle \bs{u},\pje \rangle=0$ for $j\geq 1, \,\ezj$, consequently the observations $\langle \bs{g}_{\ep_{sig}}, \pje \rangle$, $j\geq0$ are pure noise. We hence simulate $\bs{K}_{\de_{sig}}$ and, integrating the precedently computed value of $\kappa$, apply the procedure for increasing values of $\tau_{sig}$ (resp. $\tau_{op})$ until all the computed coefficients $\langle \tilde{\bs{u}},\pje\rangle $ are killed for $j\leq 3$. The results are reported in table \ref{Signal benchmark} and give $\tau_{sig}=0.9$, $\tau_{op}=0.2$.
\begin{table}
\centering
\begin{tabular}{c|cccccc}
$\kappa$ &0.3& 0.4& 0.5&0.6&0.7&0.8
\\ \hline
$Nr_{op}$ &10&9&9&8&2&0
\end{tabular}
\caption{\footnotesize Chosing of $\kappa$. $Nr_{op}$ is the average number, computed on a base of $N=10$ realizations, of levels $l\leq 10$ such that $T_{op}(\bs{K}_\de)^l \neq 0$. We have $\de=10^{-3}$.}
\label{Operator benchmark}
\end{table}

\begin{table}
\centering
\begin{tabular}{c|cccccccc}
\backslashbox{}{$\tau_{sig}$}&0.5&0.6&0.7&0.8&0.9
 \\ \hline $j=0$&3&0&3&0&0
 \\\hline $j=1$&10&6&0&0&0
 \\\hline $j=2$&20&9&2&1&0
 \\\hline $j=3$&94&22&8&4&0
\end{tabular}
\hspace{1cm}\begin{tabular}{c|ccccc}
\backslashbox{}{$\tau_{op}$}&0.1&0.2
 \\ \hline $j=0$&0&0
 \\\hline $j=1$&0&0
 \\\hline $j=2$&4&0
 \\\hline $j=3$&127&0
\end{tabular}
\caption{\footnotesize Chosing of $\tau$. For $(\de_{sig},\ep_{sig})=(\ep_{op},\de_{op})=(10^{-4},10^{-3})$ and each value of $\tau$, we computed $10$ times the described procedure and reported the average number of remaining needlet coefficients at level $j$.}
\label{Signal benchmark}
\end{table}
\renewcommand{\arraystretch}{1.2}
\begin{table}
\centering
\begin{tabular}{|c|c||cc||cc|}
\hline
\multirow{2}{*}{$\de$}&\multirow{2}{*}{$\ep$} &\multicolumn{2}{c||}{$E\|\tilde{\bs{f}}-\bs{f}\|_2$}&\multicolumn{2}{|c|}{$E\|\tilde{\bs{f}}-\bs{f}\|_\infty$}
\\\cline{3-6} 
&&\textbf{BBD}&\textbf{BND}&\textbf{BBD}&\textbf{BND}
\\\hline
\hline
\multirow{2}{*}{$3.10^{-3}$}
&$10^{-3}$&0.2214 &0.1018 &0.3877 &0.3457 
\\\cline{2-6}
&$10^{-4}$&0.1691&0.1606&0.2155&0.3377
\\\hline
\hline
\multirow{2}{*}{$10^{-3}$}
&$10^{-3}$&0.2202&0.1268&0.3846&0.2268
\\\cline{2-6}
&$10^{-4}$&0.0834&0.0595&0.1926&0.1572
\\ \hline 
\hline
\multirow{2}{*}{$10^{-4}$}
&$10^{-3}$& 0.2231&0.1257&0.3925 &0.2237 
\\\cline{2-6}
&$10^{-4}$&0.0824&0.0584&0.1924&0.1568
\\ \hline
\end{tabular}
\caption{\footnotesize Average normalized $L^2$ and $L^\infty$ loss of \textbf{BBD} and \textbf{BND}.}

\label{Mean error}
\end{table}

We compare the performances of \textbf{BBD} (with parameters taken from \cite{DHPV}) and \textbf{BND} for $\de\in\{3.10^{-3},10^{-3},10^{-4}\},\;\ep\in\{10^{-3},10^{-4}\}$. We perform the algorithm and run a Monte Carlo method over $N=20$ simulations in order to determine the mean squared error and mean $L^\infty$ error, each of whom is approximated by the discrete equivalents calculated from a uniform grid of 4096 points on $\Sph$ at each step. Results are reported in table \ref{Mean error} and confirm the shape of the obtained rates: \textbf{BND} clearly outperforms \textbf{BBD} in every situation except when the operator noise is highly predominant ($(\de,\ep)=(3.10^{-3},10^{-4})$). This was expectable since $\bs{K}$ also verifies \ref{BBD assumption} so that the rates of \cite{DHPV} are available.
\\ For particular realizations of $\bs{g}_\ep$ and $\bs{K}_\de$, we plot in figure \ref{Graphics} : the original shape of the density, and the results of the different algorithms in the form of spherical views seen 'from above'. The figures show the better adaptivity of \textbf{BND} to the 'spiky' shape of the target density.
\begin{figure}
\begin{center}
\subfigure[Target Density]
    {
       \includegraphics[width=5cm,height=5cm]{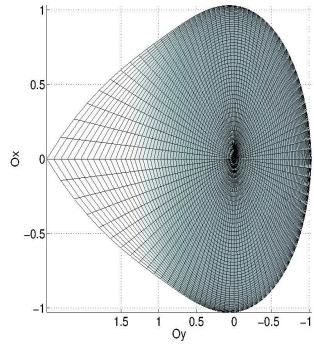}
    }
    \\    
\subfigure[\textbf{BBD}, $\ep=10^{-3}$]
    {
        \includegraphics[width=5cm,height=5cm]{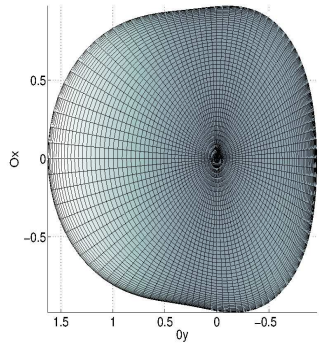}
    }    
\subfigure[\textbf{BND}, $\ep=10^{-3}$]
    {
        \includegraphics[width=5cm,height=5cm]{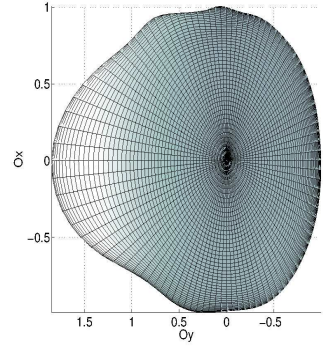}
    }
    \\
\subfigure[\textbf{BBD}, $\ep=10^{-4}$]
    {
        \includegraphics[width=5cm,height=5cm]{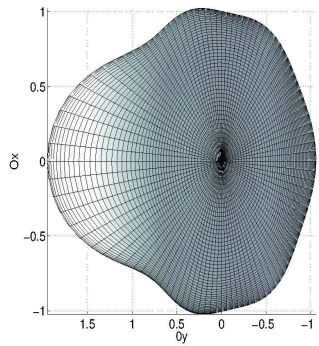}
    }    
\subfigure[\textbf{BND}, $\ep=10^{-4}$]
    {
        \includegraphics[width=5cm,height=5cm]{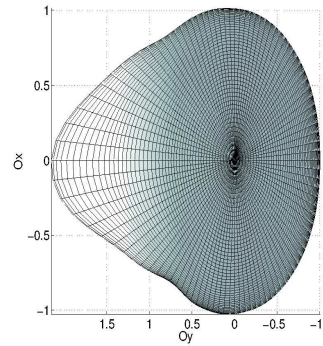}
    }
    \caption{\footnotesize Spherical view from above of the results of the two algorithms with noise level $\de=10^{-3}$}
    \label{Graphics}
    \end{center}
\end{figure}
\section{Proof of theorems \ref{Convergenge p fini} and \ref{Convergenge p infini}}
\subsection*{Preliminary lemmas}
We first establish deviation bounds on the variables $|\hbje-\bje|$ which will be useful further. We begin by the following lemma which concerns the deviations of $\|\dot{\bs{B}}^l\|_{op}$. A reference is \citet{DS}.
\begin{lemma}
There exists $\beta_0$ and $c_0$ independent from $l\in\mathbb{N}$ such that
$$\forall t\geq\beta_0,\;\PP((2l+1)^{-1/2}\|\dot{\bs{B}}^l\|_{op}>t) \leq  \exp(-c_0t(2l+1)^2)$$
A simple corrolary is the following majoration of the moments of $\|\dot{\bs{B}}^l\|_{op}$ 
$$\E[\|\dot{\bs{B}}^l\|_{op}^p] \lesssim l^{p/2}$$
\label{Operator concentration}
\end{lemma}

\begin{lemma}
\label{Neumann}
We introduce further the event
$\{\|\delta \boldsymbol{\dot B}^l\|_{\text{op}} \leq a_l\}$ with 
$a_l = \rho O_{l,\de}$ for some $0 < \rho < \tfrac{1}{2}$. 
On $\bs{A}_l \overset{\Delta}{=} \{\|(\boldsymbol{K}_{\delta,l})^{-1}\|_{\text{op}}\leq O_{l,\de}^{-1}\}$ and $\{\|\delta \boldsymbol{\dot B}^l\|_{\text{op}} \leq a_l\}$, since $a_l$ satisfies $O_{l,\de}^{-1}\, a_l = \rho < \tfrac{1}{2}$, by a usual Neumann series argument (see \citet{DHPV}), 
\begin{align*}
\|(\bs{K}_\de^l)^{-1}\|_{op}&\leq \frac{\rho}{1-\rho} \|(\bs{K}^l)^{-1}\|_{op} 
\\  \text{ and } \|(\bs{K}^l)^{-1}\|_{op}&\leq (1-\rho)^{-1} \|(\bs{K}_\de^l)^{-1}\|_{op}
\end{align*}
\end{lemma}
\begin{lemma}
Let $\overline{S_j}(\de,\ep)=\tau 2^{j\nu}\big(\ep \sqrt{|\log \ep|} \vee 2^{-j/2} \de \sqrt{|\log \de|} \big)$ with $\tau=\tau_{siq}\vee\tau_{op}$. In the setting of Theorem \ref{Convergenge p fini}, for all $j\leq J,\,\ezj$, for all $p\geq 1$
\begin{align}
\PP&(|\hbje-\bje|>\overline{S_j}(\de,\ep))\lesssim \ep^{\kappa^2} \vee\de^{\kappa^2} \label{hbje deviation}
\\\E&[|\hbje-\bje|^p]\lesssim (\ep 2^{j\nu})^p \vee (\de 2^{j(\nu-1/2)})^p \vee|\bje|^p\ind{j\geq j_0}\label{hbje expectation}
\\\E&[\displaystyle{\sup_{\ezj}}|\hbje-\bje|^p]\lesssim (j+1)^p\Big[(\ep 2^{j\nu})^p\vee (\de 2^{j(\nu-1/2)})^p\Big]\vee|\bje|^p\ind{j\geq j_0} \label{hbje expectation inf}
\end{align}
where $2^{j_0}\sim \de^{-\frac{2}{2\nu+1}}$.
\label{Lemma deviation}
\end{lemma}

\begin{proof}[Proof of Lemma \ref{Lemma deviation}]
All inequalities can be derived from the study of $\PP(|\hbje-\bje|>t)$ in each case. Recoursing to the identity 
\begin{equation} (\bs{K}_\de^l)^{-1}(\bs{K}^l \bs{f}^l+\ep \bs{\dot{W}}^l)-\bs{f}^l=
-\de (\bs{K}_\de^l)^{-1} \dot{\bs{B}}^l \bs{f}^l
+ (\bs{K}_\de^l)^{-1} \ep \bs{\dot{W}}^l
\label{Biais} 
\end{equation} which holds for every $l\in\mathbb{N}$, and using Parseval's formula, we decompose
{\small
\begin{align*}
\hbje-\bje  =&  \Sum_{l\in L_j} \big[\langle (\bs{K}_\de^l)^{-1}  \bs{g}_\ep -\bs{f}^l, \pje^l \rangle\ind{\boldsymbol{A}_l} - \langle \bs{f}^l ,\pje^l \rangle \ind{\boldsymbol{A}_l^c} \big]
\\=&  \Sum_{l\in L_j} \langle -\de (\bs{K}_\de^l)^{-1}\ind{\boldsymbol{A}_l} \dot{\bs{B}}^l \bs{f}^l, \pje^l \rangle
+ \Sum_{l\in L_j} \langle (\bs{K}_\de^l)^{-1}\ind{\boldsymbol{A}_l} \ep \bs{\dot{W}}^l , \pje^l \rangle
\\&-\Sum_{l\in L_j}\langle \bs{f}^l ,\pje^l \rangle \ind{\boldsymbol{A}_l^c}
\\ \overset{\Delta}{=}& I+II+III
\end{align*}
}
So we have to study the deviation bounds of these three terms. Term I can be decomposed as
\begin{align*}
I&= -\Sum_{l\in\bs{L}_j}\langle \de (\bs{K}_\de^l)^{-1} \dot{\bs{B}}^l\bs{f}^l,\pje^l\rangle\ind{\boldsymbol{A}_l}\big( \ind{\|\de\dot{\bs{B}}^l\|_{op}<\bs{a}_l}+\ind{\|\de\dot{\bs{B}}^l\|_{op}>\bs{a}_l} \big) \\
 &= IV+V
\end{align*}
In order to treat the term $IV$, 
we introduce the operator $$\bs{Q}_j=\Sum_{l\in L_j} (\bs{K}_\de^l)^{-1}\ind{\bs{A}_l}\ind{\|\delta \bs{\dot{B}}^l\|_{\text{op}} \leq a_l}\bs{\dot{B}}^l$$ defined for $j\leq J$. Since $\bs{K}$ and $\bs{\dot{B}}$ are both stable on every space $\mathbb{H}_l$, and since $\langle \pje, \pha \rangle=0$ if $|j-h|>1$,
{\small
\begin{align*}
IV=&|\langle \de\bs{Q}_j f,\pje \rangle|
\\=& \big|\Sum_{h=j-1,j,j+1\atop \alpha \in \mathcal{Z}_h} \langle \de\bs{Q}_j\pha,\pje \rangle \langle f,\pha \rangle  \big|
\\\leq& \big(\Sum_{h=j-1,j,j+1\atop \alpha \in \mathcal{Z}_h} | \langle \de\bs{Q}_j\pha,\pje \rangle|^{\pi'}\big)^{\frac{1}{\pi'}} \big( \Sum_{h=j-1,j,j+1\atop \alpha \in \mathcal{Z}_h}|\langle f,\pha \rangle|^{\pi}\big)^{\frac{1}{\pi}}\ind{\pi\leq 2}
\\&+\big(\Sum_{h=j-1,j,j+1\atop \alpha \in \mathcal{Z}_h} | \langle \de\bs{Q}_j\pha,\pje \rangle|^{\pi}\big)^{\frac{1}{\pi}} \big( \Sum_{h=j-1,j,j+1\atop \alpha \in \mathcal{Z}_h}|\langle f,\pha \rangle|^{\pi'}\big)^{\frac{1}{\pi'}}\ind{\pi >2}
\end{align*}
}
where we used H\"older's inequality with $\frac{1}{\pi}+\frac{1}{\pi'}=1$. Now, if $\pi\leq 2$, then $\pi'\geq 2$ and \ref{L2 behaviour of needlets} together with Proposition \ref{Needlets properties} entail
\begin{align*}
\big(\Sum_{h=j-1,j,j+1\atop \alpha \in \mathcal{Z}_h} | \langle \de\bs{Q}_j\pha,\pje \rangle|^{\pi'}\big)^{\frac{1}{\pi'}} &\leq \big(\Sum_{h=j-1,j,j+1\atop \alpha \in \mathcal{Z}_h} | \langle \de\bs{Q}_j\pha,\pje \rangle|^{2}\big)^{\frac{1}{2}}
\\&\leq \|\de ^T\bs{Q}_j\pje\|
\\&\lesssim \de 2^{j(\nu+1/2)}
\end{align*}
Moreover, since $f\in B_{\pi,r}^s$, we have $$\big( \Sum_{h=j-1,j,j+1\atop \alpha \in \mathcal{Z}_h}|\langle f,\pha \rangle|^{\pi}\big)^{\frac{1}{\pi}}\lesssim 2^{-j(s-\frac{2}{\pi}+1)}$$
If $\pi>2$, a similar argument added with the Besov embedding $B_{\pi,r}^s\in B_{\pi',r}^{s-2(\frac{1}{\pi}-\frac{1}{\pi'})}$ leads to the same bounds. Finally,
\begin{align}
\notag \PP(|IV|>t)&\leq \PP\big(\|\de ^T\bs{Q}_j\|_{op}2^{-j(s-\frac{2}{\pi}+1)}\gtrsim t\big)
\\\notag &\leq \PP(2^{-j/2}\|\bs{P}_{L_j}\bs{\dot{B}}\bs{P}_{L_j}\|_{op}\gtrsim t\de^{-1}2^{j(\nu-1/2-(s-2/\pi))})
\\&\label{deviation de} \leq \exp\big(-\frac{c_0 t^22^{2j}}{2 2^{2j(\nu-\frac{1}{2})}}\big)\ind{t\gtrsim \beta_0 2^{j(\nu-\frac{1}{2})}}
\end{align}
where we noted $\bs{P}_{L_j}$ the orthogonal projector onto $\displaystyle{\bigoplus_{l\in L_j}} \mathbb{H}_l$ and used Lemma \ref{Operator concentration}, Lemma \ref{Neumann} together with the fact that $s>\frac{2}{\pi}$. Turning to $V$, a direct application of Lemma \ref{Operator concentration} entails 
\begin{equation}
\label{PBl}
\PP(\|\de\dot{\bs{B}}^l\|_{op}>\bs{a}_l)\leq \de^{c_0\rho^2 (2l+1)^2\kappa^2}
\end{equation}
So we have
\begin{align*}
\PP(|V|>t)& \leq \PP( \de \Sum_{l\in\bs{L}_j} \| (\bs{K}_\de^l)^{-1} \dot{\bs{B}}^l\|_{op} \| \bs{f}_l\| \|\pje^l\|\ind{\bs{A}_l}\ind{\|\de\dot{\bs{B}}^l\|_{op}>\bs{a}_l}  >t) \\
\\ &\leq \Sum_{l\in\bs{L}_j} \PP (  \| (\bs{K}_\de^l)^{-1} \dot{\bs{B}}^l\|_{op} \ind{\bs{A}_l}\ind{\|\de\dot{\bs{B}}^l\|_{op}>\bs{a}_l}  >t)
\\& \lesssim \Sum_{l\in\bs{L}_j} \PP( \| (\bs{K}_\de^l)^{-1} \dot{\bs{B}}^l\|_{op}\ind{\bs{A}_l} >t  )^{1/2} \PP(\|\de\dot{\bs{B}}^l\|_{op}>\bs{a}_l)^{1/2} 
\\&\lesssim \Sum_{l\in\bs{L}_j} \PP((2l+1)^{-1/2} \|\dot{\bs{B}}^l\|_{op}> t \kappa \log^{1/2} \de)^{1/2} \de^{c_0\rho^2 (2l+1)^2\kappa^2/2}
\\ & \lesssim \de^{c_0\rho^2 2^{2j}\kappa^2/2} \Sum_{l\in\bs{L}_j} \exp( -c_0 (2l+1)^2 t^2 \kappa^2 \log\de/2)
\\ & \lesssim\de^{c_0\rho^2 2^{2j}\kappa^2/2} \exp( -c_0 2^{2j} t^2 \kappa^2 \log\de/2)
\end{align*}
Turning to the term $II$, we decompose in a similar way

$\begin{array}{rl}
II&=\Sum_{l\in\bs{L}_j}\langle \ep (\bs{K}_\de^l)^{-1} \bs{\dot{W}}^l , \pje^l  \rangle \ind{\boldsymbol{A}_l} \big( \ind{\de\|\dot{\bs{B}}^l\|_{op}\leq \bs{a}_l} + \ind{\de\|\dot{\bs{B}}^l\|_{op}> \bs{a}_l}\big) \\
&\overset{\Delta}{=} VI+VII
\end{array}$

Conditionning on $(\dot{\bs{B}}^l)_{l\in L_j}$, and applying Lemma \ref{Neumann}, we derive, for all $t>0$,

\begin{align}
\notag \PP(|VI|>t)&= \PP\Big(\big|\Sum_{l\in\bs{L}_j}\langle \ep (\bs{K}_\de^l)^{-1} \dot{\bs{W}}^l , \pje^l  \rangle\ind{\boldsymbol{A}_l} \ind{\de\|\dot{\bs{B}}^l\|_{op}\leq \bs{a}_l}\big|>t\Big)
\\&\leq \exp \big (-\frac{t^2}{2 \ep^2 2^{2j\nu}} \big)
\label{deviation ep}
\end{align}

As for $VII$, employing  Cauchy-Scwharz inequality, \ref{PBl}, and conditioning on $(\dot{\bs{B}}^l)_{l \in L_j}$, we write
{\small
\begin{align*}
\PP(|VII|>t)& = \PP(|\Sum_{l\in\bs{L}_j}\langle \ep (\bs{K}_\de^l)^{-1} \dot{\bs{W}}^l , \pje^l  \rangle \ind{\boldsymbol{A}_l}\ind{\de\|\dot{\bs{B}}^l\|_{op}> \bs{a}_l}|>t)
 \\&\leq \Sum_{l\in\bs{L}_j} \PP(|\langle \ep (\bs{K}_\de^l)^{-1} \dot{\bs{W}}^l , \pje^l  \rangle \ind{\boldsymbol{A}_l}\ind{\de\|\dot{\bs{B}}^l\|_{op}> \bs{a}_l}|>t)
\\&\leq   \Sum_{l\in\bs{L}_j} \PP( |\langle \ep (\bs{K}_\de^l)^{-1}\ind{\boldsymbol{A}_l}\dot{\bs{W}}^l , \pje^l \rangle | >t )^{1/2} \PP(\de\|\dot{\bs{B}}^l\|_{op}> \bs{a}_l)^{1/2}
 \\ & \lesssim \exp\big(-\frac{t^2\de^2|\log \de|2^{j}}{4\ep^2}\big) \de^{c_0\rho^2 2^{2j}\kappa^2/2}
\end{align*}
}
It remains to treat term $III$. We claim that 
\begin{align}\label{Term III}
\ind{\bs{A}_l^c} \leq  \ind{\|\de\bs{\dot{B}}^l\|\geq O_{l,\de} }+\ind{\|(\bs{K}^l)^{-1}\|_{op} \geq O_{l,\de}^{-1}/2 }
\end{align} (for a proof, we refer to \citet{DHPV}). Hence, 
\begin{align*}
|III| \leq & |\Sum_{l\in L_j}\langle f^l,\pje^l \rangle|\ind{\|\de\dot{\bs{B}}\|\geq O_{l,\de} }+|\Sum_{l\in L_j}\langle f^l,\pje^l \rangle|\ind{\|(\bs{K}^l)^{-1}\|_{op} \geq O_{l,\de}^{-1}/2}
 \\ \overset{\Delta}{=}& VIII+IX 
\end{align*}
As $$\{\|(\bs{K}^l)^{-1}\|_{op}>O_{l,\de}^{-1}/2\}\subset \{l>c(\de\sqrt{|\log\de|}\big)^{-\frac{1}{\nu+1/2}}\}$$
for a constant $c$ depending only on $\kappa$ and $Q_2$, we derive noting $j_0=\lfloor c(\de\sqrt{|\log\de|}\big)^{-\frac{1}{\nu+1/2}} \rfloor +1$ so that for all $j<j_0$, $l\in L_j\Rightarrow \ind{\|(\bs{K}^l)^{-1}\|_{op} \geq O_{l,\de}^{-1}/2}=0$,
\begin{align}
\label{deviation bje}
\PP(|IX|>t)\leq \ind{t<|\bje|}\ind{j\geq j_0}
\end{align}
Now, a quick application of Lemma \ref{Operator concentration} entails
\begin{align*}
\PP\big(\|\de\dot{\bs{B}}^l\|\geq O_{l,\de}\big) \leq \de^{c_0 \kappa^2 (2l+1)^2}
\end{align*}
Hence,
\begin{align*}
P(|VIII|>t)& \leq \PP\big(|\Sum_{l\in L_j} \langle \bs{f}^l, \pje^l\rangle \ind{ \|\de\dot{\bs{B}}^l\|\geq O_{l,\de}}|>t\big)
\\&\lesssim \Sum_{l\in L_j} P\big( \ind{ \|\de\dot{\bs{B}}^l\|\geq O_{l,\de}}>t\big)
\\&\lesssim \Sum_{l\in L_j} \E[\ind{ \|\de\dot{\bs{B}}^l\|\geq O_{l,\de}} \ind{t\leq 1}]
\\&\lesssim \Sum_{l\in L_j}\PP\big( \|\de\dot{\bs{B}}^l\|\geq O_{l,\de} \big)^{1/2} \ind{t\leq 1}
\\&\lesssim \de^{c_0 \kappa^2 2^{2j}/2}\ind{t\leq 1}
\end{align*}
\ref{hbje deviation} results directly from the previous deviation inequalities. Inequalities \ref{hbje expectation} and \ref{hbje expectation inf} are both applications of the well known formula \begin{align*}E[|X|^p]=\int_{u>0} pu^{p-1} \PP(|X|>u)du\leq p\int_{u>0} u^{p-1} \big(1\wedge\PP(|X|>u)\big)du
\end{align*}
Indeed, noticing that, if one takes $\kappa$ and $\tau$ large enough, the leading terms in their studies are given by \ref{deviation de}, \ref{deviation ep} and \ref{deviation bje}, inequality \ref{hbje expectation} follows immediatly. As for inequality \ref{hbje expectation inf}, we have
\begin{align*}
\E[\displaystyle{\sup_{\ezj}}|\hbje-\bje|^p]&\leq \int_{u>0} pu^{p-1} \big(1\wedge\PP(\displaystyle{\sup_{\ezj}}|\hbje-\bje|>u)\big)du
\\&\leq p\int_{u>0} u^{p-1} \big(1\wedge 2^{2j}\PP(|\hbje-\bje|>u)\big)du
\end{align*}
Moreover, considering only the terms \ref{deviation de}, \ref{deviation ep} and \ref{deviation bje}, we have
\begin{align*}
2^{2j}\PP(|\hbje-\bje|>u)\lesssim& e^{-\frac{u^2}{2\ep^22^{2j\nu}}+2j\log2}+e^{-\frac{u^2}{2\de^22^{j(2\nu-1)}}+2j\log2} 
\\&+2^{2j}\ind{u\lesssim \de 2^{j(2\nu-1)}} +2^{2j}\ind{u\leq |\bje|\ind{j\geq j_0}}
\end{align*}
which entails \ref{hbje expectation inf}.
\end{proof}
\subsection{Proof of Theorem \ref{Convergenge p fini}}
\begin{proof}
We shall only investigate the case where $p>\pi$, since for  $p\leq \pi$, we have $B_{\pi,r}^s \subset B_{p,r}^s$. The $L^p$ loss of the procedure can be decomposed as follows:
$$\E \|\tilde{\bs{f}}-\bs{f} \|_p^p \lesssim \E\|\Sum_{j \leq J} \Sum_{\eta\in \mathcal{Z}_j} \langle \tilde{\bs{f}}-\bs{f}, \pje  \rangle\pje \|_p^p + \| \Sum_{j > J} \Sum_{\eta\in \mathcal{Z}_j} \bje \pje\|_p^p $$
Since $f\in B_{\pi,r}^s$, the second term is bounded by
$$2^{-Jp\big(s-2(\frac{1}{\pi}-\frac{1}{p})\big)}$$
It is not difficult to show that $\frac{s-2(\frac{1}{\pi}-\frac{1}{p})}{x+1}$ is always larger than $\mu(x)$ (see \cite{KPP}). Hence $\mu(\nu)\leq\frac{s-2(\frac{1}{\pi}-\frac{1}{p})}{\nu+1}\leq s-2(\frac{1}{\pi}-\frac{1}{p})$ and $\mu(\nu-1/2)\leq \frac{s-2(\frac{1}{\pi}-\frac{1}{p})}{\nu+1/2}\leq 2\big(s-2(\frac{1}{\pi}-\frac{1}{p})\big)$. To bound the first term, we apply H\"older's inequality and \ref{Lp behaviour of needlets}, to write:

\begin{align*}
\E&\|\Sum_{j \leq J} \Sum_{\eta\in \mathcal{Z}_j} \langle \tilde{\bs{f}}-\bs{f}, \pje  \rangle\pje \|_p^p
\\ &\lesssim J^{p-1} \Big( \Sum_{j\leq J} \Sum_{\eta \in \mathcal{Z}_j} \E\big[|\hbje-\bje|^p\ind{|\hbje|>S_j(\de,\ep)}\big] \|\pje\|_p^p 
\\&+\Sum_{j\leq J}\Sum_{\eta \in \mathcal{Z}_j}\E\big[|\bje|^p\ind{|\hbje|\leq S_j(\de,\ep)}\big]\|\pje\|_p^p\Big)
\\&\overset{\Delta}{=}B+S
\end{align*}

The first step is to replace $S_j(\de,\ep)$ by a quantity explicitly depending on $2^{j\nu}$, namely $\overline{S_j}(\de,\ep)$. Write hence
{\small
\begin{align*}
B=& J^{p-1} \Sum_{j\leq J} \Sum_{\eta \in \mathcal{Z}_j} \E\Big[|\hbje-\bje|^p\ind{|\hbje|>S_j(\de,\ep)}\ind{l_j<+\infty}\\&\hfill \big(\ind{\|\de \bs{\dot{B}}^{l_j}\|\leq a_{l_j}}+\ind{\|\de \bs{\dot{B}}^{l_j}\|>a_{l_j}}\big)\Big] \|\pje\|_p^p
\\\leq& J^{p-1}  \Big(\Sum_{j\leq J} \Sum_{\eta \in \mathcal{Z}_j} \E\Big[|\hbje-\bje|^p\ind{|\hbje|>\overline{S_j}(\de,\ep)}\Big] \|\pje\|_p^p
\\&+\Sum_{j\leq J} \Sum_{\eta \in \mathcal{Z}_j} \E\Big[|\hbje-\bje|^{2p}\Big]^{p/2}\de^{c_0\rho^2 (22^{j}+1)^2\kappa^2/2} \|\pje\|_p^p\Big)
\end{align*}
}
where we applied Lemma \ref{Neumann}, \ref{PBl} and Cauchy-Schwartz inequality. It is clear that the second term is negligible for $\kappa$ large enough. In a similar way,
{\small
\begin{align*}
S=& J^{p-1} \Sum_{j\leq J} \Sum_{\eta \in \mathcal{Z}_j} \E\Big[|\bje|^p\ind{|\hbje|\leq S_j(\de,\ep)}\big(\ind{l_j<+\infty}+\ind{l_j=+\infty}\big)\\&\hfill \big(\ind{\|\de \bs{\dot{B}}^{l_j}\|\leq a_{l_j}}+\ind{\|\de \bs{\dot{B}}^{l_j}\|>a_{l_j}}\big)\Big] \|\pje\|_p^p
\\ \leq& J^{p-1} \Sum_{j\leq J} \Sum_{\eta \in \mathcal{Z}_j} \Big(\E\Big[|\bje|^p\ind{|\hbje|\leq \overline{S_j}(\de,\ep)}\Big]+|\bje|^p\PP\big(\|\de \bs{\dot{B}}^{l_j}\|>a_{l_j}\big)
\\&\hspace{4cm}+\E\big[|\bje|^p\ind{l_j=+\infty}\big] \Big)\|\pje\|_p^p
\end{align*}
}
Moreover, thanks to \ref{Term III},$$\ind{l_j=+\infty}\leq \ind{\bs{A}_{2^j}^c}\leq \ind{\|\de\dot{\bs{B}}^{2^j}\|\geq O_{2^j,\de} }+\ind{\|(\bs{K}^{2^j})^{-1}\|_{op} \geq O_{2^j,\de}^{-1}/2 }$$
It is clear (see the treatment of Term $III$ and \ref{PBl}) that the domining term is $$J^{p-1} \Sum_{j\leq J} \Sum_{\eta \in \mathcal{Z}_j} \E\Big[|\bje|^p\ind{|\hbje|\leq \overline{S_j}(\de,\ep)}\Big] \|\pje\|_p^p$$
Hence
\begin{align*}
\E\|\Sum_{j \leq J} \Sum_{\eta\in \mathcal{Z}_j} \langle \tilde{\bs{f}}-\bs{f}, \pje  \rangle\pje \|_p^p \lesssim J^{p-1}\big(I+II+III+IV\big)
\end{align*}
with 
{\small \begin{align*}
Bb&=\Sum_{j\leq J,\ezj} \E\big[|\hbje-\bje|^p\ind{|\hbje|>\overline{S_j}(\de,\ep)}\ind{|\bje|>\overline{S_j}(\de,\ep)/2}\big] \|\pje\|_p^p 
\\Bs&=\Sum_{j\leq J,\ezj} \E\big[|\hbje-\bje|^p\ind{|\hbje|>\overline{S_j}(\de,\ep)}\ind{|\bje|\leq \overline{S_j}(\de,\ep)/2}\big] \|\pje\|_p^p 
\\Sb&=\Sum_{j\leq J,\ezj}|\bje|^p\E\big[\ind{|\hbje|\leq \overline{S_j}(\de,\ep)}\ind{|\bje|>2\overline{S_j}(\de,\ep)}\big]\|\pje\|_p^p
\\Ss&=\Sum_{j\leq J,\ezj}|\bje|^p\E\big[\ind{|\hbje|\leq \overline{S_j}(\de,\ep)}\ind{|\bje|\leq 2\overline{S_j}(\de,\ep)}\big]\|\pje\|_p^p
\end{align*}
}
We can now treat the terms $Bs,Bb,Sb$ and $Ss$, applying \ref{Lp behaviour of needlets}, \ref{hbje deviation} and Cauchy-Schwarz inequality:
{\small
\begin{align*}
Bs &\leq J^{p-1} \Sum_{j\leq J}\Sum_{\ezj} \E\big[|\hbje-\bje|^p\ind{|\hbje-\bje|>\overline{S_j}(\de,\ep)/2}\big] \|\pje\|_p^p
\\&\leq J^{p-1} \Sum_{j\leq J}\Sum_{\ezj} \E[|\hbje-\bje|^{2p}]^{1/2}\PP(|\hbje-\bje|>\overline{S_j}(\de,\ep)/2)^{1/2} \|\pje\|_p^p
\\&\lesssim  J^{p-1} \Sum_{j\leq J}\Sum_{\ezj} \big((\ep 2^{j\nu})^p\vee(\de 2^{j(\nu-1/2)})^p\vee |\bje|^p\ind{j\geq j_0} \big) 2^{jp} \big(\ep^{\tau^2}\vee\de^{\tau^2}\big)
\end{align*}
}
Moreover,
\begin{align*}
Sb&\leq J^{p-1} \Sum_{j\leq J}\Sum_{\ezj}|\bje|^p \PP(|\hbje-\bje|>\overline{S_j}(\de,\ep)) \|\pje\|_p^p
\\&\lesssim J^{p-1}\big(\ep^{\tau^2}\vee\de^{\tau^2}\big)
\end{align*}
since $f\in B_{p,r}^{s-2(1/\pi-1/p)}$. Hence in both cases the rate of convergence is smaller than what is claimed for sufficiently large $\tau$. Turning to $Bb$ and $Ss$, we write, for all $z,z'\geq 0$,
{\small
\begin{align*}
Bb\lesssim& J^{p-1}\Sum_{j\leq J}\Sum_{\ezj} \E[|\hbje-\bje|^p] \ind{|\bje|>\overline{S_j}(\de,\ep)/2} \|\pje\|^p
\\ \lesssim&  \Sum_{j\leq J}  \Sum_{\ezj} \big((\ep 2^{j\nu})^p \vee (\de 2^{j(\nu-1/2)})^p\vee |\bje|^p\ind{j\geq j_0}  \big) \ind{|\bje|>  \overline{S_j}(\de,\ep)/2}\|\pje\|^p
\\\lesssim& J^{p-1} \big(\ep\sqrt{|\log \ep|}\big)^{p-z} \Sum_{j\leq J} 2^{j[\nu(p-z)+p-2]} \Sum_{\ezj} |\bje|^z
\\&+J^{p-1}\big(\de\sqrt{|\log \de|}\big)^{p-z'} \Sum_{j\leq J} 2^{j[(\nu-1/2)(p-z')+p-2]} \Sum_{\ezj} |\bje|^{z'}
\\&+J^{p-1} 2^{-j_0p\big(s-2(\frac{1}{\pi}-\frac{1}{p})\big)}
\end{align*}
}and
\begin{align*}
Ss \lesssim& J^{p-1}\Sum_{j\leq J} \Sum_{\ezj} |\bje|^z \Big(\ind{|\bje|\leq 2\tau 2^{j\nu}\ep\sqrt{|\log \ep|}}
\\&\hspace{2cm}+ \ind{|\bje|\leq 2\tau 2^{j(\nu-1/2)}\de\sqrt{|\log \de|}} \Big) \|\pje\|_p^p
\\\lesssim&J^{p-1}\big(\ep\sqrt{|\log \ep|}\big)^{p-z} \Sum_{j\leq J} 2^{j[\nu(p-z)+p-2]} \Sum_{\ezj} |\bje|^z\|\pje\|_p^p
\\&+J^{p-1}\big(\de\sqrt{|\log \de|}\big)^{p-z'} \Sum_{j\leq J} 2^{j[(\nu-1/2)(p-z')+p-2]} \Sum_{\ezj} |\bje|^{z'}\|\pje\|_p^p
\end{align*}
We already bounded $2^{-j_0p\big(s-2(\frac{1}{\pi}-\frac{1}{p})\big)}\sim \de^{p\frac{s-2(\frac{1}{\pi}-\frac{1}{p})}{\nu+1/2}}$, so in both cases we have the same term to bound. This term can be further writen as $R(\ep,\nu,z)+R(\de,\nu-1/2,z')$ where
\begin{align*}
R(x,y,z)=J^{p-1}\big(x\sqrt{|\log x|}\big)^{p-z} \Sum_{j\leq J} 2^{j[y(p-z)+p-2]} \Sum_{\ezj} |\bje|^{z}\|\pje\|_p^p
\end{align*}
We only give a brief overview of the treatment of $R$; a detailed one is present in \citet{KPP}. First, we split $R$ as follows
\begin{align*}
R(x,y,z)& =J^{p-1}\Big[\big(x\sqrt{|\log x|}\big)^{p-z_1} \Sum_{j\leq J_0} 2^{j[y(p-z_1)+p-2]} \Sum_{\ezj} |\bje|^{z_1}\|\pje\|_p^p
\\&+\big(x\sqrt{|\log x|}\big)^{p-z_2} \Sum_{j>J_0} 2^{j[y(p-z_2)+p-2]} \Sum_{\ezj} |\bje|^{z_2}\|\pje\|_p^p\Big]
\end{align*}
 where $z_1,z_2,J_0$ are to determine. Consider first the case where $s\geq (y+1)(\frac{p}{\pi}-1)$. Note $q=p\frac{y+1}{s+y+1}$. Taking $z_2=\pi$, $z_1=\tilde{q}<q$ and $ 2^{J_0\frac{p}{q}(y+1)}\sim (x\sqrt{|\log x|} )^{-1}$ entail  $$ R(x,y,J_0)\lesssim (\log x)^{p-1} (x\sqrt{|\log x|})^{p-q}$$
 \\which is the desired bound. Now consider the case where $s< (y+1)(\frac{p}{\pi}-1)$ and note $q=p\frac{y+1-2/p}{y+1+s-2/\pi}$. Take $z_1=\pi$, $z_2=\tilde{q}>q$ and $2^{J_0\frac{p}{q}(y+1-2/p)}\sim (x\sqrt{|\log x|} )^{-1}$. We obtain $$ R(x,y,J_0)\lesssim (\log\ep)^{p-1} (x \sqrt{|\log x|})^{p-q}$$
which ends the proof.
\end{proof}
\subsection{Proof of Theorem \ref{Convergenge p infini}}
\begin{proof}
Write similarly
$$\|\tilde{\bs{f}}-\bs{f}\|_\infty\leq \E\|\Sum_{j\leq J} \sum_{\ezj} \big(\hbje-\bje\big)\pje\|_\infty + \|\Sum_{j>J}\Sum_{\ezj} \bje \pje\|_\infty$$
The second term can be handed as before. We decompose the first term in the following way, using \ref{Lp behaviour of needlets} for $p=\infty$, and applying the same sketch of proof as in theorem \ref{Convergenge p fini},

$\begin{array}{rl}
\E\|\Sum_{j\leq J} \sum_{\ezj} \big(\hbje-\bje\big)\pje\|_\infty&\lesssim \Sum_{j\leq J} \E \displaystyle{\sup_{\ezj}}|\hbje-\bje| 2^j
\\&\leq Bb+Bs+Sb+Ss
\end{array}$

with 
\begin{align*}
Bb&=\Sum_{j\leq J} 2^j\E\big[\displaystyle{\sup_{\ezj}}|\hbje-\bje|\ind{|\hbje|>\overline{S_j}(\de,\ep)}\ind{|\bje|>\overline{S_j}(\de,\ep)/2}\big] 
\\Bs&=\Sum_{j\leq J}2^j \E\big[\displaystyle{\sup_{\ezj}}|\hbje-\bje|\ind{|\hbje|>\overline{S_j}(\de,\ep)}\ind{|\bje|\leq \overline{S_j}(\de,\ep)/2}\big]
\\Sb&=\Sum_{j\leq J}2^j\displaystyle{\sup_{\ezj}}|\bje|\E\big[\ind{|\hbje|\leq \overline{S_j}(\de,\ep)}\ind{|\bje|>2\overline{S_j}(\de,\ep)}\big]
\\Ss&=\Sum_{j\leq J}2^j\displaystyle{\sup_{\ezj}}|\bje|\E\big[\ind{|\hbje|\leq \overline{S_j}(\de,\ep)}\ind{|\bje|\leq 2\overline{S_j}(\de,\ep)}\big]
\end{align*}
We have, using inequatlity \ref{hbje expectation inf},
{\small
\begin{align*}
Bb&\leq \Sum_{j\leq J}2^j \E\displaystyle{\sup_{\ezj}}|\hbje-\bje|\ind{|\bje|>\overline{S_j}(\de,\ep)/2}
\\&\leq \Sum_{j\leq J}2^j \ind{\exists \ezj,\,|\bje|\geq \overline{S_j}(\de,\ep)/2}2^j  \E\displaystyle{\sup_{\ezj}}|\hbje-\bje|
\\&\lesssim \Sum_{j\leq J}2^j \ind{\exists \ezj,\,|\bje|\geq \overline{S_j}(\de,\ep)/2} (j+1)\big(\ep2^{j\nu}\vee\de2^{j(\nu-1/2)}\big)\vee|\bje|\ind{j\geq j_0}
\\&\lesssim 2^{J_1(\nu+1)}(J_1+1)\ep+2^{I_1(\nu+3/2)}(I_1+1)\de +\Sum_{j\geq j_0} 2^j |\bje|
\end{align*}
}
where $J_1$ is chosen so that, for $j\geq J_1$, $|\bje|\leq \tau \ep \sqrt{|\log \ep|}2^{j\nu}/2$. We can take for example (see \cite{KPP}) $J_1$ verifying, for a certain constant $B$, $$2^{J_1}=B \big(\ep \sqrt{|\log \ep|}\big)^{-(s+\nu+1-2/\pi)^{-1}}$$ Similarly, taking $$2^{I_1}=C \big(\de \sqrt{|\log \de|}\big)^{-(s+\nu+1/2-2/\pi)^{-1}}$$ for a certain constant $C$ implies $|\bje|\leq \tau \de \sqrt{|\log \de|}2^{j(\nu-1/2)}/2$ for all $j\leq I_1$. The term $\Sum_{j\geq j_0} 2^j |\bje|$ is easily treated. This finally leads to the rate 
\begin{align*}
Bb&\lesssim |\log \ep|\ep^{\mu'(2)}\vee |\log \de|\de^{\mu'(1)}
\\Ss&\leq  \Sum_{j\leq J}2^j \displaystyle{\sup_{\ezj}}|\bje| \ind{|\bje|\leq 2\overline{S_j}(\de,\ep)}
\\&\lesssim \Big[\Sum_{j\leq J_1}2^j \ep\sqrt{|\log \ep|} 2^{j\nu}+\Sum_{j> J_1}2^j|\bje|\Big]
\\& \hspace{4cm} \vee\Big[\Sum_{j\leq I_1}2^j\de\sqrt{|\log \de|} 2^{j(\nu-1/2)}+\Sum_{j> I_1}2^j|\bje|\Big]
\end{align*}
which gives the proper rate of convergence.
Turning to $Bs$ and $Sb$, we write, using inequalities \ref{hbje deviation} and \ref{hbje expectation inf}
{\small
\begin{align*}
\\Bs&\leq \Sum_{j\leq J}2^j \E \big[\displaystyle{\sup_{\ezj}}|\hbje-\bje|\ind{|\hbje-\bje|>\overline{S_j}(\de,\ep)/2}\big]
\\&\leq \Sum_{j\leq J}2^j \E[\displaystyle{\sup_{\ezj}}|\hbje-\bje|^2]^{1/2} \PP(\exists \ezj,\,|\hbje-\bje|>\overline{S_j}(\de,\ep)/2)^{1/2}
\\&\lesssim \Sum_{j\leq J}2^j \Big[\big(j+1\big)\big(\ep2^{j\nu}\vee\de2^{j(\nu-1/2)}\big)\vee|\bje|\ind{j\geq j_0}\Big]\Big[2^{2j}(\ep^{\tau^2}\vee\de^{\tau^2})\Big]^{1/2}
\end{align*}
}
Now apply inequality \ref{hbje deviation} and the fact that $|\bje|\lesssim 2^{-j}$ to derive

\begin{align*}
Sb&\leq \Sum_{j\leq J}2^j \E \big[\displaystyle{\sup_{\ezj}}|\bje|\ind{|\hbje-\bje|>\overline{S_j}(\de,\ep)}\big]
\\&\lesssim \Sum_{j\leq J} 2^{2j}\PP\big(|\hbje-\bje|>\overline{S_j}(\de,\ep)\big)
\\&\lesssim  \Sum_{j\leq J} 2^{2j}(\ep^{\tau^2}\vee\de^{\tau^2})
\end{align*}
It is clear that for a well chosen $\tau$ these terms are smaller than the announced rates.

\end{proof}
\subsubsection*{Acknowledgements}
The author would like to thank Dominique Picard for numerous and fruitful discussions and suggestions.

\bibliographystyle{plainnat}
\bibliography{BiblioV}
\end{document}